\documentclass[a4paper,12pt,reqno]{amsart}
\usepackage{latexsym}
\usepackage{latexsym}
\usepackage{delarray}
\usepackage{amssymb,amsmath,amsfonts,amsthm,mathrsfs}
\usepackage{hyperref}

\numberwithin{equation}{section}
\usepackage{scalerel}

\newcommand\reallywidehat[1]{\arraycolsep=0pt\relax%
\begin{array}{c}
\stretchto{
  \scaleto{
    \scalerel*[\widthof{\ensuremath{#1}}]{\kern-.5pt\bigwedge\kern-.5pt}
    {\rule[-\textheight/2]{1ex}{\textheight}} 
  }{\textheight} %
}{0.5ex}\\           
#1\\                 
\rule{-1ex}{0ex}
\end{array}
}

\usepackage{hyperref}
\renewcommand\eqref[1]{(\ref{#1})} 

 \newtheorem{thm}{Theorem}[section]
 \newtheorem{cor}[thm]{Corollary}
 \newtheorem{lem}[thm]{Lemma}

 \newtheorem{defn}[thm]{Definition}
 \newtheorem{rem}[thm]{Remark}
 \numberwithin{equation}{section}
\newcommand{\half}{\frac{1}{2}}

\newcommand{\sdual}{\frac{1}{p}+\frac{1}{q}=1}
\newcommand{\ene}{\mathbb{N}}

\newcommand{\ar}{\mathbb{R}}
\newcommand{\ce}{\mathbb{C}}

\newcommand{\arn}{{\mathbb{R}}^n}

\newcommand{\efee}{\mathcal{F}}

\newcommand{\bi}{\begin{itemize}}
\newcommand{\ei}{\end{itemize}}
\newcommand{\be}{\begin{enumerate}}
\newcommand{\ee}{\end{enumerate}}
\newcommand{\beq}{\begin{equation}}
\newcommand{\eq}{\end{equation}}

\newcommand{\lap}{\mathcal{L}_G}
\newcommand{\lapk}{\mathcal{L}_{G/K}}
\newcommand{\cdxi}{\ce^{d_{\xi}\times d_{\xi}}}

\def\jp#1{{\left\langle{#1}\right\rangle}}

\def\Rep{{{\rm Rep}}}

\DeclareMathOperator{\Lip}{Lip}
\DeclareMathOperator{\Ob}{O}

\DeclareMathOperator{\Tr}{Tr}

\def\Ghz{{\widehat{G}}_0}
\def\Gh{{\widehat{G}}}

\def\HS{{\mathtt{HS}}}

\def\T{{{\mathbb T}^1}}

\def\Z{{{\mathbb Z}}}
\def\SU2{{{\rm SU(2)}}}
\def\SO3{{{\rm SO(3)}}}
\def\lapsu2{{{\mathcal L}_\SU2}}
\newcommand{\jpb}{\langle\xi\rangle}
\def\whf{\widehat{f}}

\begin{document}
\title[Titchmarsh theorems for H\"older-Lipschitz functions]{Titchmarsh theorems for Fourier transforms of H\"older-Lipschitz functions on compact homogeneous manifolds}

\author[Radouan Daher]{Radouan Daher}


\address{Department of Mathematics\\
University of Hassan II\\
B.P. 5366 Maarif, Casablanca\\
Morocco}

\email{rjdaher024@gmail.com}

\thanks{The second author was supported by the
Leverhulme Research Grant RPG-2014-02.
The third author was supported by
 the EPSRC Grant EP/K039407/1.
 No new data was collected or generated during the course of research.}


\author[Julio Delgado]{Julio Delgado}

 
\address{Department of Mathematics\\
Imperial College London\\
180 Queen's Gate, London SW7 2AZ\\
United Kingdom}

\email{j.delgado@imperial.ac.uk}

\author{Michael Ruzhansky}

\address{Department of Mathematics\\
Imperial College London\\
180 Queen's Gate, London SW7 2AZ\\
United Kingdom}

\email{m.ruzhansky@imperial.ac.uk}



\date{\today}

\begin{abstract}
In this paper we extend classical Titchmarsh theorems on the Fourier transform of H\"older-Lipschitz functions to the setting of 
 compact homogeneous manifolds. As an application, we derive a Fourier multiplier theorem for $L^2$-H\"older-Lipschitz spaces on compact Lie groups. We also derive conditions and a characterisation for Dini-Lipschitz classes on compact homogeneous manifolds in terms of the behaviour of their Fourier coefficients.
\end{abstract}

\maketitle
\section{Introduction}

The studies of the convergence and of the rate of decay of Fourier coefficients are among the most classical problems in Fourier analysis. 
  Starting from the Riemann-Lebesgue theorem relating the integrability of a function on the torus $\T$ and the convergence of its Fourier coefficients, through the Hausdorff-Young inequality relating the integrability of a function and of its Fourier transform.
 Furthermore one can relate the smoothness of a function on the torus and the rate of decay  of its Fourier coefficients: indeed if $f\in C^k(\T)$ for some $k\geq 1$,   one has
  \[\widehat{f}(j)=o(|j|^{-k}),\]
that is, $|j|^k|\widehat{f}(j)|\rightarrow 0$ as $j\rightarrow\infty.$ An extension of such statement below the index $k=1$ can be obtained in the setting of  H\"older regularity. If $f$ satisfies the H\"older condition of order $0<\alpha\leq 1$ then 
\[\widehat{f}(j)=\Ob(|j|^{-\alpha}).\]
This time we only get big-O (Landau's notation). For the case $\alpha>\half$ this estimate can be improved as a consequence of the classical Bernstein theorem (see Bernstein \cite{be:fa} or Zygmund \cite[vol. I, p. 240]{az:st}): if $f$
 satisfies the H\"older condition of order $\alpha>\half$ then its Fourier series converges absolutely. It was also established by Bernstein that the index $\half$ is sharp. An introduction to the  classical topics  discussed above can be found in \cite{steins:fa}.

\medskip
The Hausdorff-Young inequality on a locally compact abelian group $G_1$ states that $\whf\in L^q$ if  $f\in L^p(G_1)$ for $1\leq p\leq 2$ and $\sdual$. 
 In \cite{tich:b1}, Edward Charles Titchmarsh studied the problem of how much this fact can be strengthened if $f$ additionally satisfies a Lipschitz condition in the case of $G_1=\ar$. In fact, Titchmarsh showed that the information on $\whf$ can be significantly improved.  In \cite{myou:lpli}, \cite{myou:li1}, \cite{myou:li2}, M. Younis extended the Titchmarsh theorems to the case of the circle and to the very specific case of compact groups, namely, the $0$-dimensional groups.
 
 \medskip
The aim of this paper is to extend the classical Titchmarsh theorems to the setting of general compact homogeneous manifolds. In particular, this includes higher dimensional tori, compact Lie groups, real, complex or quaternionic spheres, projective spaces, and many others. As an application of such an extension, we derive a Fourier multiplier theorem for $L^2$-Lipschitz spaces.

\medskip
The analysis of various spaces and functional inequalities on compact Lie groups and compact homogeneous manifolds, despite its own interest and long history, has also a range of applications linking the local properties of functions with the global spectral properties of the group and to its representation theory, see e.g. Pesenson \cite{Pesenson:Besov-2008} for Bernstein-Nikolskii inequalities and links to the approximation theory, \cite{Pesenson-PW} for links to frames construction and Geller and Pesenson \cite{Geller-Pesenson:frames-Besov} for Besov spaces, Kushpel \cite{Kushpel-cubature} for cubature formulae, \cite{dr:gevrey} for Gevrey spaces with further links to weakly hyperbolic partial differential equations involving Laplacian and sub-Laplacian,
\cite{NRT-Pisa} for Nikolskii inequality and Besov spaces, as well as further references in these papers.
The convergence of the Fourier series of functions on compact Lie groups has a long history, see e.g., to mention only very few, Taylor \cite{Taylor:Fourier}, Sugiura \cite{sugiura:p1}, and some review in Faraut's book \cite{Faraut:book}. We can also refer to \cite{Benke1} for conditions for Lipschitz functions on compact totally disconnected groups, \cite{Platonov05} on symmetric spaces of rank 1,
\cite{QY1} for Fourier series on Vilenkin groups (compact, metrisable, $0$-dimensional, abelian groups), and results for Lipschitz multipliers in \cite{QY3}, \cite{QY2}, \cite{Quek85} and \cite{Bloom81} on compact $0$-dimensional, locally compact abelian, and abelian $0$-dimensional groups, respectively.

We note that for convolution operators (on compact Lie groups) the rate of decay of Fourier coefficients of the kernel is also related to the Schatten-von Neumann properties of the operator, see e.g.
\cite{dr13:schatten} and \cite{dr13a:nuclp}.
 
We also note that anisotropic H\"older spaces (with respect to the group structure) have been also analysed on the Heisenberg group
in \cite{FS-1974} (see also \cite{Folland-1975}) with applications to partial differential equations, and some Fourier transform description was considered in \cite{Younis-Heisenberg}.
Anisotropic H\"older spaces on more general Carnot groups were recently applied in \cite{RS-AM} to questions in the potential theory.
 
 \medskip
Thus, the first Titchmarsh theorem that we deal with (Titchmarsh \cite[Theorem 84]{tich:b1}) in its version on the torus (\cite[Theorem 2.6]{myou:lpli}) is recalled as Theorem A below.
First, for $0<\alpha\leq 1$ and $1<p\leq \infty$, we define the space $\Lip_{\T}(\alpha ;p)$ by
\beq\Lip_{\T}(\alpha;p):=\{f\in L^p(\T):\|f(h+\cdot )-f(\cdot)\|_{L^p(\T)}=\Ob(h^{\alpha}) \mbox{ as }h\rightarrow 0\}.\label{lip1}\eq

\medskip
\noindent {\bf{Theorem A.}}{\textit{ Let $0<\alpha\leq 1$ and $1<p\leq 2$. If $f\in \Lip_{\T}(\alpha;p)$, then $\widehat{f}\in \ell^{\beta}(\Z)$ for $\frac{p}{p+\alpha p-1}<\beta\leq q$, $\sdual$.}}\\

In order to illustrate the improvement in the conclusion $\widehat{f}\in \ell^{\beta}(\Z)$ in Theorem A in comparison to the Hausdorff-Young theorem, consider for instance $f\in \Lip_{\T}(\half;2)$, then one has $\frac{p}{p+\alpha p-1}=1,$ and thus $\widehat{f}\in \ell^{\beta}(\Z)$ for all $\beta\in (1,2].$

\smallskip
In Theorem A$'$ and Theorem \ref{Ti1} we extend Theorem A to the setting of general compact homogeneous manifolds, also slightly sharpening the conclusion.

\smallskip
We point out that our proof of Theorem A$'$ (through Theorem \ref{Ti1}) for compact homogeneous manifolds also simplifies the one by Younis in some aspects already on the circle $\T$, particularly avoiding the use of combinatorial arguments for summations by parts. The fact that the lower bound $\frac{p}{p+\alpha p-1}$ is sharp can be proved by means of the function $$f(x)=\sum\limits_{n=1}^{\infty}\frac{e^{in \log n}}{n^{\half+\alpha}}e^{inx} \quad(0<\alpha<1)$$ introduced by Hardy and Littlewood. Indeed, one has 
 $f\in \Lip_{\T}(\alpha;2)$ but $\widehat{f}\notin \ell^{\frac{2}{2\alpha+1}}(\Z)$ (see Zygmund \cite[vol 1, p. 243]{az:st}).

\medskip
 
The second Titchmarsh theorem (Titchmarsh \cite[Theorem 85]{tich:b1}) gives the necessary and sufficient conditions on the remainder of the series of Fourier coefficients for $L^2$-Lipschitz functions, i.e. for the case $p=2$. In the case of the circle (\cite[Theorem 2.17]{myou:lpli}) it can be stated as follows:

\medskip
\noindent {\bf{Theorem B.}}{\textit{ Let $0<\alpha\leq 1$ and $f\in L^2(\T)$. Then 
$f\in \Lip_{\T}(\alpha;2)$ if and only if }
\[\sum\limits_{|j|\geq N}|\widehat{f}(j)|^2=\Ob(N^{-2\alpha}) \mbox{ \em as } N\rightarrow\infty.\]

We will now briefly outline our main results in the setting of general compact homogeneous manifolds. If $K$ is a closed subgroup of a compact Lie group $G$, the quotient space $G/K$ can be canonically identified with an analytic manifold $M$. Henceforth $M$ will denote a compact homogeneous manifold. In the case when $K=\{e\}$ is the identity of the group $G$, we have $M=G$ is a compact Lie group itself. To simplify the notation in the sequel we will identify representations of $G$ with their equivalence classes.

We denote by $\Gh$ the unitary dual of $G$, and 
by $\Ghz$ the subset of $\Gh$ consisting of representations $\xi$ of
class I with respect to the subgroup $K$: this means that $[\xi]\in\Ghz$ if $\xi$ has at least one non-zero invariant vector with respect to $K$, i.e.
 that for some non-zero $a$ in the representation space of $\xi$ we have $\xi(h)a=a$ for all $h\in K$. We will also denote by $k_{\xi}$  the number of invariant vectors of the representation $\xi$ with respect
to $K$. 

In particular, if $K=\{e\}$ so that $G/K=G$ is a compact Lie group, then $k_\xi=d_\xi$ is the dimension of the representation $\xi\in\Gh$. If the subgroup $K$ is massive, that is if there is exactly one invariant vector for each class I representation, then we have $k_\xi=1$ for all $\xi\in\Gh_0$,
and $k_\xi=0$ for all $\xi\in\Gh\backslash\Gh_0$.

In order to extend the definition \eqref{lip1} on the circle to a compact homogeneous manifold $M=G/K$ we will employ the geodesic distance on $M$. For the distance from the `unit' element $eK$ to $h\in M$ we will write $|h|=d(h,e)$. This convention is not restrictive since the analysis of our problems is local.

\begin{defn}\label{lip2a} Let $0<\alpha\leq 1$ and $1\leq p\leq\infty$. We define the space $\Lip_{G/K}(\alpha ;p)$ by
\[\Lip_{G/K}(\alpha;p):=\{f\in L^p(G/K):\|f(h\cdot)-f(\cdot)\|_{L^p(G/K)}=\Ob(|h|^{\alpha}) \mbox{ as }|h|\rightarrow 0\},\]
that is, by the condition that
\begin{equation}\label{EQ:ll1}
\int_{G/K}|f(hx)-f(x)|^pdx=\Ob(|h|^{\alpha p})  \; \mbox{ as } \; |h|\rightarrow 0,
\end{equation}
for $1\leq p<\infty$, with a natural modification for $p=\infty$.
\end{defn}
The space $\Lip_{G/K}(\alpha;p)$ endowed with the norm 
\[\|f\|_{\Lip_{G/K}(\alpha;p)}:=\sup\limits_{|h|\neq 0}|h|^{-\alpha}\|f(h\cdot)-f(\cdot)\|_{L^p(G/K)}\]
becomes a Banach space.

\medskip
If $f\in \Lip_{G/K}(\alpha;p)$ we say that $f$ is $(\alpha;p)$-Lipschitz. Since $K$ acts on $G/K$ from right, it is natural to consider left shifts $f(h\cdot)$ with respect to $h$ in Definition \ref{lip2a}.
However, if $K=\{e\}$ and our homogeneous space is the Lie group $G$ itself, one can also consider right shifts by $h$, and to define the right-Lipschitz space by
taking $f(\cdot h)$ instead of  $f(h\cdot)$ in the Definition \ref{lip2a}, namely, by the condition
\begin{equation}\label{EQ:ll2}
\int_{G/K}|f(xh)-f(x)|^pdx=\Ob(|h|^{\alpha p}) \; \mbox{ as } \; |h|\rightarrow 0,
\end{equation}
for $1\leq p<\infty$, with a natural modification for $p=\infty$.

Some remarks regarding the equivalence between these two notions in the case of groups are included at the end of Section \ref{SEC:Prelim} in Remark \ref{equiv1}.

\medskip
On the other hand, for $1\leq p\leq\infty$ we can also define $p-$Lipschitz spaces on $G$ via localisations. Given $f\in L^p(G/K)$, we say that 
\begin{equation}\label{EQ:lipp1}
f\in \widetilde{\Lip_{G/K}}(\alpha;p)
\;\textrm{ if }\; \chi^*(\chi f)\in \Lip_{\arn}(\alpha;p)
\end{equation}
 for all coordinate chart cut-off functions $\chi$, where $\chi^*$ denotes the pull-back by $\chi$. We can observe that
 for $p=\infty$, with natural modifications of the definitions, we have
 \begin{equation}\label{EQ:lip}
\widetilde{\Lip_{G/K}}(\alpha;\infty)
= {\Lip_{G/K}}(\alpha;\infty).
\end{equation}
In view of Lemma \ref{ewrt1}, it is enough to show this for compact Lie groups.
The inclusion $\widetilde{\Lip_{G/K}}(\alpha;\infty)
\subset{\Lip_{G/K}}(\alpha;\infty)$ is obvious due to the compactness of $G/K$.
The converse also follows if we take $h=\exp(h_j X_j)$, $j=1,\ldots,\dim G$,  for small real numbers $h_j$, with $\exp:\mathfrak g\to G$ the exponential mapping near $e$, and $\{X_j\}_{j=1}^{\dim G}$ a basis of $\mathfrak g$, e.g. orthonormalised with respect to the Killing form.
Using that the push-forwards $\chi^*$ by coordinate chart maps have non-vanishing Jacobians, compactness, and adapting the argument above in the other direction, the equality \eqref{EQ:lip} can be extended to any $p$, namely, we have
 \begin{equation}\label{EQ:lip2}
\widetilde{\Lip_{G/K}}(\alpha;p)
= {\Lip_{G/K}}(\alpha;p),\quad 1\leq p\leq\infty.
\end{equation}

\medskip
The main results of this work are
Theorem A$'$ and Theorem B$'$ below, extending the statements of
Theorem A and Theorem B, respectively, from the circle to a compact homogeneous manifolds
$G/K$.

\medskip
\noindent {\bf{Theorem A$'$.}}{\textit{ 
Let $G/K$ be a compact homogeneous manifold of dimension $n$.
Let $0<\alpha\leq 1$, $1<p\leq 2$, and let $q$ be such that $\sdual$.
Let $f\in \Lip_{G/K}(\alpha;p)$. 
Then 
$$\reallywidehat{(I-\lapk)^{\half}f}\in \ell^{\beta}(\Ghz) \;\textrm{  for }\; \frac{n}{\alpha+n-\frac{n}{p}-1}\leq \beta\leq q,$$ 
where $\lapk$ is the Laplacian on $G/K$. \\
Consequently, we also have 
$$\whf\in \ell^{\gamma}(\Ghz)\;\textrm{ for }\;\frac{np}{\alpha p+np-n}<\gamma\leq q.$$ }}

\medskip
Here for a sequence $\sigma(\xi)\in\ce^{d_\xi\times d_\xi}$
the norm $\ell^\beta(\Ghz)$ for $1\leq\beta<\infty$ is defined by
\begin{equation}\label{EQ:lpdef}
\|\sigma\|_{\ell^{\beta}(\Ghz)}:=\sum\limits_{[\xi]\in\Ghz}d_{\xi} k_\xi\left(\frac{\|\sigma(\xi)\|_{\HS}}{\sqrt{k_{\xi}}}\right)^{\beta},
\end{equation}
where $\|\sigma(\xi)\|_{\HS}=\sqrt{\Tr(\sigma(\xi)\sigma(\xi)^*)}$ is the usual Hilbert-Schmidt norm.
This is a natural family of spaces on $\Ghz$, extending the $\ell^2$-norm appearing in the  Plancherel identity in \eqref{EQ:Planch}, also satisfying the Hausdorff-Young inequalities on $G/K$ and many other functional analytic properties.
These spaces were analysed in \cite[Section 2]{NRT-Pisa} in the setting of compact homogeneous manifolds, extending the corresponding definition on the compact Lie groups introduced earlier in
\cite[Section 10.3.3]{rt:book}, in which case we would have $\Ghz=\Gh$ and $k_\xi=d_\xi$: see
Definition \ref{DEF:lpG} for this case.
The Fourier coefficients of functions on $G/K$, although matrices in $\ce^{d_\xi\times d_\xi}$, have a lot of zeros, see \eqref{EQ:clft1} and \eqref{EQ:clft2}. This explains the factor $\sqrt{k_\xi}$ appearing in the denominator in \eqref{EQ:lpdef}. The spaces $\ell^\beta(\Ghz)$ are interpolation spaces, with embeddings into another known family of $\ell^\beta$-spaces defined using Schatten norms, we refer to
\cite[Section 2]{NRT-Pisa} for a description of functional analytic properties of these spaces and to 
\cite[Section 2.1.4]{FR-book} for the embedding properties.

\medskip
Clearly (see Remark \ref{REM:circle}) Theorem A$'$ implies Theorem A when $G=\T$ and $K=\{0\}$, so that $G/K=\T$.
Since both Theorem A$'$ and Theorem B$'$ absorb Titchmarsh theorems in the case of the torus they are also sharp.

\medskip
\noindent {\bf{Theorem B$'$.}}{\textit{ Let $0<\alpha\leq 1$ and $f\in L^2(G/K)$. Then 
$f\in \Lip_{G/K}(\alpha;2)$ if and only if }
\begin{equation}\label{EQ:b1}
\sum\limits_{\xi\in\Ghz,\langle\xi\rangle\geq N}d_{\xi}\|\widehat{f}(\xi)\|_{\HS}^2=\Ob(N^{-2\alpha}) \mbox{ \em as } N\rightarrow\infty.
\end{equation}

\medskip
Here $\jp{\xi}$ are the eigenvalues of the elliptic operator $(I-\lapk)^{1/2}$ corresponding to the representations $\xi\in\Ghz$, where $\lapk$ is the Laplacian on $G/K$: namely, we have
$(I-\lapk)^{1/2}\xi_{ij}=\jp{\xi}\xi_{ij}$ for all $\xi\in\Ghz$, $1\leq i\leq d_\xi$ and $1\leq j\leq k_\xi$.
The choice of the range of the indices $i,j$ happens in this way since these are the only representation coefficients interacting with the Fourier coefficients of functions on $G/K$, 
see \eqref{EQ:clft1} and \eqref{EQ:clft2}. In the case of Lie groups the quantity $\jp{\xi}$ is also explained in detail in \eqref{EQ:Lap-lambda}.

\medskip
Theorem B$'$ implies Theorem B since the Plancherel formula on $G/K$ takes the form
\begin{equation}\label{EQ:Planch}
\sum\limits_{\xi\in\Ghz}d_{\xi}\|\widehat{f}(\xi)\|_{\HS}^2=\|f\|_{L^2(G/K)}^2,
\end{equation}
so that the left hand side of \eqref{EQ:b1} is the remainder of the Fourier series decomposition of $f\in L^2(G/K)$, see also \eqref{EQ:Parseval}.

\medskip
In view of Lemma \ref{ewrt1}, Theorem A$'$ and Theorem B$'$ follow from their respective versions on compact Lie groups formulated in Theorem \ref{Ti1}  and Theorem \ref{dulem}, respectively.

\medskip
As an application of the obtained characterisations, we will show a Fourier multiplier theorem for H\"older/Lipschitz spaces on compact Lie groups. Here, the Fourier multiplier $A$ with symbol $a(\xi)\in \ce^{d_{\xi}\times d_{\xi}}$ for each $\xi$ acts by multiplication by  $a(\xi)$ on the Fourier transform side, i.e. by
$$
\widehat{Af}(\xi)=a(\xi)\widehat{f}(\xi).
$$
The matrix $a(\xi)$ arises as a linear transformation $a(\xi):\ce^{d_\xi}\to \ce^{d_\xi}$, where $\ce^{d_\xi}$ is the representation space of $\xi$ after fixing some (any) basis. We denote by
$\|a(\xi)\|_{op}=\|a(\xi)\|_{\mathcal L(\ce^{d_\xi},\ce^{d_\xi})}$ the operator norm of this linear transformation, with $\ce^{d_\xi}$ equipped with the usual Euclidean distance. If $\mathcal{L}_G$ 
is the Laplacian on $G$ (or the Casimir element of the universal enveloping algebra), we denote by $\jp{\xi}$ the eigenvalue of the elliptic pseudo-differential operator $(I-\mathcal{L}_G)^{1/2}$ corresponding to $\xi$, i.e. 
$$
(I-\mathcal{L}_G)^{1/2}\xi_{ij}(x)=\jp{\xi}\xi_{ij}(x)\qquad\textrm{ for all } 1\leq i,j\leq d_{\xi},
$$
see \eqref{EQ:jpxi} for more explanation.

\begin{cor}\label{c1q0} 
Let $0\leq\gamma<1$ and let $a:\Gh\rightarrow\bigcup\limits_{d\in\ene}\ce^{d\times d}$ be such that $a(\xi)\in \ce^{d_{\xi}\times d_{\xi}}$ for each $\xi\in\Gh$ and
\[\|a(\xi)\|_{op}\leq C\jp{\xi}^{-\gamma}. \]
Then the Fourier multiplier $A$ by symbol $a$, 
\[A:\Lip_{G}(\alpha;2)\rightarrow\Lip_{G}(\alpha+\gamma;2) \]
is bounded for all $\alpha$ such that $0<\alpha<1-\gamma$.
\end{cor}

For example, the matrix symbol of the Bessel potential $A=(I-\mathcal{L}_G)^{-\gamma/2}$ is $a(\xi)=\jp{\xi}^{-\gamma} I_{d_\xi}$, where $I_{d_\xi}\in\mathbb C^{d_\xi\times d_\xi}$ is the identity matrix, so that $\|a(\xi)\|_{op}=\jp{\xi}^{-\gamma}.$

\medskip
 
A similar result to Corollary \ref{c1q0} is true on homogeneous manifolds $G/K$, although in this case, in view of  \eqref{EQ:clft1} and \eqref{EQ:clft2}, the symbol $a(\xi)$ may be assumed to be zero outside the first $k_\xi\times k_\xi$ block, i.e. we would take
$a(\xi)=0$ for $\xi\not\in\Ghz$, and
$a(\xi)_{ij}=0$ for all $\xi\in\Ghz$ and for $i>k_\xi$ or $j>k_\xi$.

\smallskip
The multiplier theorem in Corollary \ref{c1q0} complements other known multiplier theorems on compact Lie groups, e.g. the Mihlin multiplier theorem \cite{Ruzhansky-Wirth:Lp-FAA,Ruzhansky-Wirth:Lp-Z}, H\"ormander multiplier theorem \cite{Fischer-Lp}, or more general results for the boundedness of pseudo-differential operators in $L^p$-spaces \cite{DR-Jussieu}.

\smallskip
In Section \ref{SEC:dini-li} we discuss analogues of the obtained results for the Dini-Lipschitz classes. 
The Dini-Lipschitz classes are a logarithmic extension on Lipschitz classes and these spaces have been analysed in different settings, e.g. in low dimensions \cite{Younis-DL}, or using Helgason transforms on rank 1 symmetric spaces in \cite{Daher-DL-Helg,FBE-DL}.
Our analysis in this part of the paper is based on a logarithmic extension of the Duren lemma allowing inclusion of the log-terms in the characterisations.


\section{Preliminaries}  
\label{SEC:Prelim}

In this section we recall some basic facts on the Fourier analysis and  the notion of global matrix-symbols for pseudo-differential operators on compact Lie groups and more generally on compact homogeneous manifolds. We refer to  \cite{rt:book} and \cite{rt:groups} for a comprehensive account of such topics.  

Since the proof of our main theorems will be reduced to the case of compact Lie groups as it will be explained at the end of this section, it will be convenient for our purposes to start with the setting of compact Lie groups.
 Given a compact Lie group $G$, we equip it with the normalised Haar measure $\mu\equiv dx$ on the Borel $\sigma$-algebra associated to the topology of 
the smooth manifold $G$. The Lie algebra of $G$ will be denoted by $\mathfrak{g}$. We also denote by $\widehat{G}$ the set of equivalence classes of continuous irreducible unitary 
representations of $G$ and by $\Rep(G)$ the set of all such representations. Since $G$ is compact, the set $\widehat{G}$ is discrete.  
For $[\xi]\in \widehat{G}$, by choosing a basis in the representation space of $\xi$, we can view 
$\xi$ as a matrix-valued function $\xi:G\rightarrow \ce^{d_{\xi}\times d_{\xi}}$, where 
$d_{\xi}$ is the dimension of the representation space of $\xi$. 
By the Peter-Weyl theorem the collection
$$
\left\{ \sqrt{d_\xi}\,\xi_{ij}: \; 1\leq i,j\leq d_\xi,\; [\xi]\in\Gh \right\}
$$
is an orthonormal basis of $L^2(G)$.
If $f\in L^1(G)$ we define its global Fourier transform at $\xi$ by 
\begin{equation}\label{EQ:FG}
\mathcal F_G f(\xi)\equiv \widehat{f}(\xi):=\int_{G}f(x)\xi(x)^*dx.
\end{equation}
Thus, if $\xi$ is a matrix representation, 
we have $\widehat{f}(\xi)\in\ce^{d_{\xi}\times d_{\xi}} $. The Fourier inversion formula is a consequence
 of the Peter-Weyl theorem:
\beq \label{EQ:FGsum}
f(x)=\sum\limits_{[\xi]\in \widehat{G}}d_{\xi} \Tr(\xi(x)\widehat{f}(\xi)).
\eq
Given a sequence of matrices $a(\xi)\in\mathbb C^{d_\xi\times d_\xi}$, we can define
\begin{equation}\label{EQ:FGi}
(\mathcal F_G^{-1} a)(x):=\sum\limits_{[\xi]\in \widehat{G}}d_{\xi} \Tr(\xi(x) a(\xi)),
\end{equation}
 where the series can be interpreted in a distributional sense or absolutely depending on the growth of 
(the Hilbert-Schmidt norms of) $a(\xi)$. For a further discussion we refer the reader to \cite{rt:book}.

For each $[\xi]\in \widehat{G}$, the matrix elements of $\xi$ are the eigenfunctions for the Laplacian $\mathcal{L}_G$ 
(or the Casimir element of the universal enveloping algebra), with the same eigenvalue which we denote by 
$-\lambda^2_{[\xi]}$, so that
\begin{equation}\label{EQ:Lap-lambda}
-\mathcal{L}_G\xi_{ij}(x)=\lambda^2_{[\xi]}\xi_{ij}(x)\qquad\textrm{ for all } 1\leq i,j\leq d_{\xi}.
\end{equation} 
The weight for measuring the decay or growth of Fourier coefficients in this setting is 
\begin{equation}\label{EQ:jpxi}
\jp{\xi}:=(1+\lambda^2_{[\xi]})^{\half},
\end{equation} 
 the eigenvalues of the elliptic first-order pseudo-differential operator 
$(I-\mathcal{L}_G)^{\half}$.
The Parseval identity takes the form 
\begin{equation}\label{EQ:Parseval}
\|f\|_{L^2(G)}= \left(\sum\limits_{[\xi]\in \widehat{G}}d_{\xi}\|\widehat{f}(\xi)\|^2_{\HS}\right)^{1/2},\quad
\textrm{ where }
\|\widehat{f}(\xi)\|^2_{\HS}=\Tr(\widehat{f}(\xi)\widehat{f}(\xi)^*),
\end{equation}
which gives the norm on 
$\ell^2(\widehat{G})$. 

For a linear continuous operator $A$ from $C^{\infty}(G)$ to $\mathcal{D}'(G) $ 
we define  its {\em matrix-valued symbol} $\sigma(x,\xi)\in\cdxi$ by 
\begin{equation}\label{EQ:A-symbol}
\sigma(x,\xi):=\xi(x)^*(A\xi)(x)\in\cdxi,
\end{equation}
where $A\xi(x)\in \cdxi$ is understood as $(A\xi(x))_{ij}=(A\xi_{ij})(x)$, i.e. by 
applying $A$ to each component of the matrix $\xi(x)$.
Then one has (\cite{rt:book}, \cite{rt:groups}) the global quantization
\begin{equation}\label{EQ:A-quant}
Af(x)=\sum\limits_{[\xi]\in \widehat{G}}d_{\xi}\Tr(\xi(x)\sigma(x,\xi)\widehat{f}(\xi))\equiv\sigma(X,D)f(x),
\end{equation}
in the sense of distributions, and the sum is independent of the choice of a representation $\xi$ from each 
equivalence class 
$[\xi]\in \widehat{G}$. If $A$ is a linear continuous operator from $C^{\infty}(G)$ to $C^{\infty}(G)$,
the series \eqref{EQ:A-quant} is absolutely convergent and can be interpreted in the pointwise 
sense. The symbol $\sigma$ can be interpreted as a matrix-valued
function on $G\times\widehat{G}$.
We refer to \cite{rt:book}, \cite{rt:groups} for the consistent development of this quantization
and the corresponding symbolic calculus. If the operator $A$ is left-invariant then its symbol
$\sigma$ does not depend on $x$. We often simply call such operators invariant.

In the sequel we will also need the Taylor expansion of functions on compact Lie groups.
For $f\in C^\infty(G)$, we have
\begin{equation}\label{EQ:TE}
f(x)=\sum\limits_{|\alpha|\leq N-1} D^{(\alpha)}f(e)q_{\alpha}(x)+O(|x|^N),
\end{equation}
for some invariant differential operators $D^{(\alpha)}$ of order $|\alpha|$, for an admissible 
family of functions $q_{\alpha}$, with $|x|$ denoting the geodesic distance from $x$ to $e$, see
\cite[Section 10.6]{rt:book}.


\medskip
We note that Definition \ref{lip2a} in the case of compact Lie groups can be stated in the following form by taking the geodesic distance on $G$ from the unit element $e$ to $h\in G$:

\begin{defn}\label{lip2} Let $G$ be a compact Lie group. Let $0<\alpha\leq 1$ and $1\leq p\leq\infty$. We define the space $\Lip_G(\alpha ;p)$ by
\[\Lip_G(\alpha;p):=\{f\in L^p(G):\|f(h\cdot)-f(\cdot)\|_{L^p(G)}=\Ob(|h|^{\alpha }) \mbox{ as }|h|\rightarrow 0\},\]
that is, 
\[\int_G|f(hx)-f(x)|^pdx=\Ob(|h|^{\alpha p}) \;\textrm{ as }\; |h|\rightarrow 0,\]
for $1\leq p<\infty$, with a natural modification for $p=\infty$.
\end{defn}

The space $\Lip_G(\alpha;p)$ endowed with the norm 
\[\|f\|_{\Lip_G(\alpha;p)}:=\sup\limits_{|h|\neq 0}|h|^{-\alpha}\|f(h\cdot)-f(\cdot)\|_{L^p(G)}\]
becomes a Banach space for all $1\leq p\leq\infty$.

For the analysis of the Fourier transform we will require the following  function spaces on the unitary dual $\Gh$. We refer the reader to \cite[Section 10.3.3]{rt:book} for the basic properties of such spaces.
\begin{defn} \label{DEF:lpG}
For $0< p<\infty$, we will write $\ell^p(\Gh)$ for the space of all $H=H(\xi)\in\ce^{d_\xi\times d_\xi}$ such that
\[\|H\|_{\ell^p(\Gh)}:=\left(\sum\limits_{[\xi]\in\Ghz}d_{\xi}^{p(\frac{2}{p}-\frac{1}{2})}\|H(\xi)\|_{\HS}^p\right)^{\frac 1p}<\infty.\]
If $1\leq p<\infty$ the quantity  $\|H\|_{\ell^p(\Gh)}$ defines a norm and $\ell^p(\Gh)$ endowed with it becomes a Banach space. If $0<p<1$ we can associate a Fr\'echet metric and the associated space becomes a complete metric space.
\end{defn}

We record some asymptotic properties on a compact Lie group $G$ of dimension $n$, that will be of use on several occasions: asymptotically as $\lambda\to\infty$  we have
\begin{equation}\label{EQ:dims1}
\sum_{\jp{\xi}\leq\lambda}d_{\xi}^{2} \jp{\xi}^{r n}\asymp \lambda^{(r+1)n}
\; \textrm{ for }\; r>-1,
\end{equation}
and
\begin{equation}\label{EQ:dims2}
\sum_{\jp{\xi}\geq\lambda}d_{\xi}^{2} \jp{\xi}^{r n}\asymp \lambda^{(r+1)n}\; \textrm{ for }\; r<-1.
\end{equation}
These properties follow from the Weyl spectral asymptotic formula for the elliptic pseudo-differential operator $I-\mathcal L_G$, and we can refer to \cite{ar:pq1} for their proof.
We also note the convergence criterion
\begin{equation}\label{EQ:conv}
\sum_{\jp{\xi}\leq\lambda}d_{\xi}^{2} \jp{\xi}^{-s}<\infty\;\textrm{ if and only if }\; s>n,
\end{equation}
see \cite{dr:gevrey} for the proof.

\begin{rem}\label{equiv1}
The notions of left and right-Lipschitz in \eqref{EQ:ll1} and \eqref{EQ:ll2} are related in the case of compact Lie groups in the following way.  We write $i(x)=x^{-1}$ and define $f^{\cdot}(x)=f\circ i (x)=f(x^{-1})$. We note that 
\[f(xh)-f(x)=f^{\cdot}(h^{-1}x^{-1})-f^{\cdot}(x^{-1})\]
and since $|h|\simeq |h^{-1}|$ we have obtained that $f$ is a right-Lipschitz of type $(\alpha;p)$ function if and only if $f^{\cdot}$ is left-Lipschitz of type $(\alpha;p)$.


Our statements will be stated for left-Lipschitz $(\alpha;p)$. A look at the arguments in the proofs of our results shows that all the statements also hold for right-Lipschitz $(\alpha;p)$ type. At the level of compact Lie groups both concepts agree when $p=2$ as a consequence of Theorem \ref{Ti2}: indeed, the criterion \eqref{t2mg} is the same for both left- or right-Lipschitz functions.
\end{rem}

We shall now explain how the proof of Theorems A$'$ and B$'$ can be reduced to the case of compact Lie groups. For a function $f\in C^\infty(G/K)$, its canonical lifting $\widetilde{f}$ is defined by $\widetilde{f}(yk)=f(y)$ for all $k\in K$, so that $\widetilde{f}$ is constant on the right cosets. 

\begin{lem} \label{ewrt1} 
We have $\widetilde{f}\in \Lip_{G/K}(\alpha; p)$ if and only if $f\in \Lip_{G/K}(\alpha; p)$.
\end{lem}
\begin{proof}  A simple calculation shows
\begin{align*} \int_G|\widetilde{f}(hx)-\widetilde{f}(x)|^pdx &=\int_{G/K}\left(\int_{K}|f(hyk)-f(yk)|^pdk\right)dy\\
&=\int_{G/K}|f(hy)-f(y)|^pdy,
\end{align*}
with equality, 
if the Haar measure is normalised in such a way that the measure of $K$ is equal to one.
\end{proof}

The Fourier coefficients of canonical liftings satisfy 
\begin{equation}\label{EQ:clft1}
\widehat{\widetilde{f}}(\xi)=0 \;\textrm{ for all }\; \xi\not\in\Ghz,
\end{equation}
and 
\begin{equation}\label{EQ:clft2}
\widehat{\widetilde{f}}(\xi)_{ij}=0 \;\textrm{ for all }\; \xi\in\Ghz,\; i>k_\xi,
\end{equation}
see e.g. Vilenkin \cite{Vilenkin:BK-eng-1968}.
We refer to \cite{dr:gevrey} or \cite{NRT-Pisa} for more details on the Fourier analysis on compact homogeneous manifolds.

\section{Titchmarsh theorems for Fourier transforms of Lipschitz functions}  
\label{SEC:ftlip}

In this section we prove our results. We start with a lemma on Fourier transforms 
that we will apply in the proof of the extension of Theorem A. The notion of global symbol discussed in Section \ref{SEC:Prelim} will be useful in the proof of the first main result. Henceforth $G$ will denote a compact Lie group of dimension $n$.

\begin{lem}\label{le1w} Let $H:\Gh\rightarrow\bigcup\limits_{d\in\ene}\ce^{d\times d}$ be such that $H(\xi)\in \ce^{d_{\xi}\times d_{\xi}}$ for each $\xi$. Let $1\leq\beta_0<\infty$. Then 
\[\langle\xi\rangle H(\xi)\in \ell^{\beta_0}(\Gh) \implies H\in \ell^{\beta}(\Gh),\]
for all $\frac{n\beta_0}{\beta_0+n}<\beta<\infty$.
\end{lem}
\begin{proof} 
Since $\langle\xi\rangle\geq 1$, we have $\langle\xi\rangle H(\xi)\in\ell^{\beta_0}(\Gh)$ implies that $H\in\ell^{\beta_0}(\Gh)$, and hence   $H\in\ell^{\beta}(\Gh)$
 for all $\beta\geq\beta_0$. So we need to prove the case $\beta<\beta_0$, which we now assume.
 
 \medskip
 For $\beta<\beta_0$ we can write
\begin{align*}\|H\|_{\ell^{\beta}(\Gh)}^{\beta}=\sum\limits_{[\xi]\in\Gh}d_{\xi}^{2}\left(\frac{\|H(\xi)\|_{\HS}}{\sqrt{d_{\xi}}}\right)^{\beta}&=\sum\limits_{[\xi]\in\Gh}d_{\xi}^{2}\langle\xi\rangle^{-\beta}\left(\frac{\langle\xi\rangle\|H(\xi)\|_{\HS}}{\sqrt{d_{\xi}}}\right)^{\beta}\\
&=\sum\limits_{[\xi]\in\Gh}d_{\xi}^{2\frac{\beta}{\beta_0}}\left(\frac{\langle\xi\rangle\|H(\xi)\|_{\HS}}{\sqrt{d_{\xi}}}\right)^{\beta}d_{\xi}^{2(1-\frac{\beta}{\beta_0})}\langle\xi\rangle^{-\beta}\\
&=\sum\limits_{[\xi]\in\Gh}a_{\xi}b_{\xi},
\end{align*}
where $a_{\xi}=d_{\xi}^{2\frac{\beta}{\beta_0}}\left(\frac{\langle\xi\rangle\|H(\xi)\|_{\HS}}{\sqrt{d_{\xi}}}\right)^{\beta}$ and $b_{\xi}=d_{\xi}^{2(1-\frac{\beta}{\beta_0})}\langle\xi\rangle^{-\beta}.$
By \eqref{EQ:conv} and the H\"older inequality applied to the last sum for the indices $\frac{\beta_0}{\beta}, \frac{\beta_0}{\beta_0-\beta}$ we obtain
\begin{align*}\sum\limits_{[\xi]\in\Gh}d_{\xi}^{2}\left(\frac{\|H(\xi)\|_{\HS}}{\sqrt{d_{\xi}}}\right)^{\beta}&\leq\left( \sum\limits_{[\xi]\in\Gh}d_{\xi}^{2}\left(\frac{\langle\xi\rangle\|H(\xi)\|_{\HS}}{\sqrt{d_{\xi}}}\right)^{\beta_0}\right)^{\frac{\beta}{\beta_0}}\times
\\
&\quad\times \left( \sum\limits_{[\xi]\in\Gh}d_{\xi}^{2(1-\frac{\beta}{\beta_0})\frac{\beta_0}{\beta_0-\beta}}\langle\xi\rangle^{-\beta(\frac{\beta_0}{\beta_0-\beta})}\right)^{(1-\frac{\beta}{\beta_0})}\\
&=\|\langle\cdot\rangle H(\cdot)\|_{\ell^{\beta_0}(\Gh)}^{\beta}\left( \sum\limits_{[\xi]\in\Gh}d_{\xi}^2\langle\xi\rangle^{-\beta(\frac{\beta_0}{\beta_0-\beta})}\right)^{(1-\frac{\beta}{\beta_0})}\\
&\leq C\left( \sum\limits_{[\xi]\in\Gh}d_{\xi}^2\langle\xi\rangle^{-\beta(\frac{\beta_0}{\beta_0-\beta})}\right)^{(1-\frac{\beta}{\beta_0})}<\infty ,
\end{align*}
provided that $\frac{\beta\beta_0}{\beta_0-\beta}>\dim G=n.$ For the convergence of the series in the last inequality we have used \eqref{EQ:conv}.
 To finish the proof we just note that the condition $\frac{n\beta_0}{\beta_0+n}<\beta$ is equivalent to $\frac{\beta\beta_0}{\beta_0-\beta}>n.$
\end{proof}

We now state our first main result which extends Theorem A to general compact Lie groups and consequently, in view of Lemma \ref{ewrt1}, to compact homogeneous manifolds.

\begin{thm}\label{Ti1} 
Let $G$ be a compact Lie group of dimension $n$.
Let $0<\alpha\leq 1$, $1<p\leq 2$, and let $q$ be such that $\sdual$. 
Let $f\in \Lip_G(\alpha;p)$.
Then 
$$\reallywidehat{(I-\lap)^{\half}f}\in \ell^{\beta}(\Gh)\,\textrm{ for } \;\frac{n}{\alpha+n-\frac{n}{p}-1}\leq \beta\leq q.$$
Consequently, 
$$\whf\in \ell^{\gamma}(\Gh)\;\textrm{ for }\;\frac{np}{\alpha p+np-n}<\gamma\leq q.$$ 
\end{thm}
\begin{proof} For $f\in L^p(G)$ we will denote by $f_h$ the function $f_h(x)=f(hx)$.
 We first note that
\begin{align*} \widehat{f_h}(\xi)&=\int_G f(hx)\xi(x)^*dx\\
&=\int_G f(hx)\xi(hx)^*\xi(h)dx\\
&=\widehat{f}(\xi)\xi(h).
\end{align*}
Hence the Fourier transform of $f_h-f$ is given by
 \[\efee(f_h-f)(\xi)=\whf(\xi)(\xi(h)-I_{d_{\xi}}),\]
 where $ I_{d_{\xi}}$ is the identity matrix in $\ce^{d_\xi\times d_\xi}$.
By the Hausdorff-Young inequality we get
\beq\|\whf(\xi)(\xi(h)-I_{d_{\xi}})\|_{\ell^q(\Gh)}^q\leq \|f_h-f\|_{L^p(G)}^q=O(|h|^{\alpha q}).\label{hy1}\eq
We now estimate from below $\|\whf(\xi)(\xi(h)-I_{d_{\xi}})\|_{\ell^q(\Gh)}^q$. To do so we start with the terms $\|\whf(\xi)(\xi(h)-I_{d_{\xi}})\|_{\HS}^q$ for every $\xi$ according to the definition of the norm $\|\cdot\|_{\ell^q(\Gh)}$.\\

We observe that 
\begin{align}\|\whf(\xi)(\xi(h)-I_{d_{\xi}})\|_{\HS}^2 &=\Tr\left(\whf(\xi)(\xi(h)-I_{d_{\xi}})(\xi(h)^*-I_{d_{\xi}})\whf(\xi)^*\right)\nonumber\\
&=\Tr\left((\xi(h)-I_{d_{\xi}})(\xi(h)^*-I_{d_{\xi}})\whf(\xi)^*\whf(\xi)\right).\label{iqf1x}
\end{align}
We now apply Taylor expansions of type \eqref{EQ:TE} to the terms $(\xi(h)-I_{d_{\xi}}),\,(\xi(h)^*-I_{d_{\xi}})$. For each entry in the matrices they are of the form
\[\xi(h)_{ij}-\delta_{ij}=\sum\limits_{|\alpha|=1}D^{(\alpha)}\xi_{ij}(e)q_{\alpha}(h)+O(|h|^2),\]
where $q_{\alpha}\in C^{\infty}(G)$  vanishes at $e$ of order $1$ and $e$ is its isolated zero, i.e., there exist constants $C_1, C_2>0$ such that
\[C_1|y|\leq |q_{\alpha}(y)|\leq C_2|y|.\] 

In the matrix form we can write 
\beq\xi(h)-I_{d_{\xi}}=\sum\limits_{|\alpha|=1}D^{(\alpha)}\xi(e)q_{\alpha}(h)+O(|h|^2).
\label{migtl}\eq
To estimate \eqref{iqf1x} we plug \eqref{migtl} into it. Then
\begin{multline*}
\Tr\left((\xi(h)-I_{d_{\xi}})(\xi(h)^*-I_{d_{\xi}})\whf(\xi)^*\whf(\xi)\right)\\
=\Tr\left(\sum\limits_{|\alpha|=1, |\beta|=1}D^{(\alpha)}\xi(e)D^{(\beta)}\xi(e)^*\whf(\xi)^*\whf(\xi)q_{\alpha}(h) q_{\beta}(h)\right)+O(|h|^3).
\end{multline*}
We observe that for each pair $\alpha,\beta$ the matrix  $D^{(\alpha)}\xi(e)D^{(\beta)}\xi(e)^*+D^{(\beta)}\xi(e)D^{(\alpha)}\xi(e)^*$ entering this sum, is self-adjoint for every $\xi$, thus it can be written in the form
\[U^{-1}\Lambda_{\alpha,\beta,\xi}U=D^{(\alpha)}\xi(e)D^{(\beta)}\xi(e)^*+D^{(\beta)}\xi(e)D^{(\alpha)}\xi(e)^*,\]
where $\Lambda_{\alpha,\beta,\xi}$ is diagonal and $U$ is unitary. Hence 
\[\Lambda_{\alpha,\beta,\xi}=U(\sigma_{D^{(\alpha)}}\sigma^*_{D^{(\beta)}}+\sigma_{D^{(\beta)}}\sigma_{D^{(\alpha)}}^*)U^{-1},\]
and we can deduce that  $\sum\limits_{|\alpha|=1,|\beta|=1}\Lambda_{\alpha,\beta,\xi}$ is a diagonal positive second order elliptic symbol since $D^{(\alpha)}$'s can be chosen to give a basis of the Lie algebra of $G$, see \cite[(10.25)]{rt:book}. The ellipticity is preserved by the action of an unitary matrix, due to the identities
\[\|B\|_{op}=\|BV\|_{op}=\|WB\|_{op},\]
where $B,V,W\in C^{d_{\xi}\times d_{\xi}}$ with $V,W$ unitary.

\medskip
Therefore, grouping terms with respect to the pair $(\alpha,\beta)$ and $(\beta,\alpha)$, we have
\begin{align*}
&\Tr\left(\sum\limits_{|\alpha|=1, |\beta|=1}D^{(\alpha)}\xi(e)D^{(\beta)}\xi(e)^*\whf(\xi)^*\whf(\xi)q_{\alpha}(h) q_{\beta}(h)\right) \\
&=\Tr\left(\sum\limits_{|\alpha|=1,|\beta|=1}\Lambda_{\alpha,\beta,\xi}[U\whf(\xi)^*][U\whf(\xi)^*]^* q_{\alpha}(h)q_{\beta}(h)\right)\\
&=\sum\limits_{i=1}^{d_{\xi}}\left(\sum\limits_{|\alpha|=1,|\beta|=1}(\Lambda_{\alpha,\beta,\xi})_{ii}([U\whf(\xi)^*][U\whf(\xi)^*]^*)_{ii}q_{\alpha}(h)q_{\beta}(h)\right)\\
&\geq C|h|^2\jpb^2\sum\limits_{i=1}^{d_{\xi}}\left(([U\whf(\xi)^*][U\whf(\xi)^*]^*)_{ii}\right)\\
&= C|h|^2\jpb^2 \Tr\left(([U\whf(\xi)^*][U\whf(\xi)^*]^*)\right)\\
&= C|h|^2\jpb^2\|\whf(\xi)\|_{\HS}^2.
\end{align*}

We have proved that
\beq\|\whf(\xi)(\xi(h)-I_{d_{\xi}})\|_{\HS}^2\geq C|h|^2\langle\xi\rangle^2\|\whf(\xi)\|_{\HS}^2
\label{main1}\eq
for small enough $|h|$.
By \eqref{hy1}  we have
$$
\sum\limits_{[\xi]\in\Gh, \langle\xi\rangle\leq \frac{1}{|h|}}d_{\xi}^{q(\frac{2}{q}-\frac{1}{2})}\|\widehat{f_h}(\xi)-\whf(\xi)\|_{\HS}^q=O(|h|^{\alpha q}),
$$
and hence by \eqref{main1} we obtain
\begin{multline*}
|h|^{q}\sum\limits_{[\xi]\in\Gh, \langle\xi\rangle\leq \frac{1}{|h|}}d_{\xi}^{q(\frac{2}{q}-\frac{1}{2})}\langle\xi\rangle^q\|\whf(\xi)\|_{\HS}^q \\
\lesssim \sum\limits_{[\xi]\in\Gh, \langle\xi\rangle\leq \frac{1}{|h|}}d_{\xi}^{q(\frac{2}{q}-\frac{1}{2})}\|\widehat{f_h}(\xi)-\whf(\xi)\|_{\HS}^q\leq 
O(|h|^{\alpha q}).
\end{multline*} 
Therefore, we obtain
\beq
\sum\limits_{[\xi]\in\Gh, \langle\xi\rangle\leq \frac{1}{|h|}}d_{\xi}^{q(\frac{2}{q}-\frac{1}{2})}\langle\xi\rangle^q\|\whf(\xi)\|_{\HS}^q
=O(|h|^{(\alpha-1)q}).
\label{main2}
\eq

Now, for $\beta\leq q$ and $N\in\ene$ we write 
$$
\Phi(N)=\sum\limits_{[\xi]\in\Gh, \langle\xi\rangle\leq N}d_{\xi}^{\beta(\frac{2}{\beta}-\frac{1}{2})}\langle\xi\rangle^{\beta}\|\whf(\xi)\|_{\HS}^{\beta},
$$
so that the analysis for $|h|\rightarrow 0$ will be translated to $N\rightarrow \infty$.  In order to apply the H\"older inequality and for the forthcoming analysis it is convenient to write $\Phi(N)$ in the form
\[\Phi(N)=\sum\limits_{[\xi]\in\Gh, \langle\xi\rangle\leq N}d_{\xi}^{2}\left(\frac{\langle\xi\rangle\|\whf(\xi)\|_{\HS}}{\sqrt{d_{\xi}}}\right)^{\beta}=\sum\limits_{[\xi]\in\Gh, \langle\xi\rangle\leq N}a_{\xi}b_{\xi},\]
where $a_{\xi}=d_{\xi}^{2(\frac{\beta}{q})}\left(\frac{\langle\xi\rangle\|\whf(\xi)\|_{\HS}}{\sqrt{d_{\xi}}}\right)^{\beta}$ and $b_{\xi}=d_{\xi}^{2(1-\frac{\beta}{q})}$.\\

Then, by H\"older inequality, \eqref{main2} and \eqref{EQ:dims1} with $r=0$, we get
\begin{align*}
\Phi(N)&\leq \left(\sum\limits_{[\xi]\in\Gh, \langle\xi\rangle\leq N}d_{\xi}^{2}\left(\frac{\langle\xi\rangle\|\whf(\xi)\|_{\HS}}{\sqrt{d_{\xi}}}\right)^q\right)^{\frac{\beta}{q}}\left(\sum\limits_{[\xi]\in\Gh, \langle\xi\rangle\leq  N}d_{\xi}^2\right)^{1-\frac{\beta}{q}}\\
&=O(N^{(1-\alpha)\beta}) N^{n(1-\frac{\beta}{q})}.\\
\end{align*}
Therefore  
\beq\sum\limits_{[\xi]\in\Gh, \langle\xi\rangle\leq N}d_{\xi}^{2}\left(\frac{\langle\xi\rangle\|\whf(\xi)\|_{\HS}}{\sqrt{d_{\xi}}}\right)^{\beta}=O(N^{(1-\alpha)\beta+n(1-\frac{\beta}{q})}).\label{miqj3}\eq

Since $\reallywidehat{(I-\lap)^{\half}f}(\xi)=\langle\xi\rangle\whf (\xi)$, we have proved that $\reallywidehat{(I-\lap)^{\half}f}\in \ell ^{\beta}(\Gh)$ provided that $(1-\alpha)\beta+n(1-\frac{\beta}{q})\leq 0$, a condition which is equivalent to $\beta\geq\frac{n}{\alpha+n-\frac{n}{p}-1}$ as we have assumed. This proves the first conclusion.

\medskip
Now, an application of Lemma \ref{le1w} to $H(\xi)=\whf (\xi)$ shows that $\whf\in\ell^{\gamma}(\Gh)$ for $\gamma >\frac{n\beta_0}{n+\beta_0}$ with 
$\beta_0=\frac{n}{\alpha+n-\frac{n}{p}-1}$. Since 
\[\frac{n\left(\frac{n}{\alpha+n-\frac{n}{p}-1}\right)}{\frac{n}{\alpha+n-\frac{n}{p}-1}+n}=\frac{np}{\alpha p+np-n}\]
we conclude the proof.
\end{proof}

\begin{rem}\label{REM:circle}
In the case $G=\T$ we observe that Theorem \ref{Ti1} for the first conclusion takes the form:
\[f\in \Lip_{\T}(\alpha;p)\implies \reallywidehat{(I-\Delta)^{\half}f}\in \ell^{\beta}(\Z)\]
for $\frac{p}{p\alpha-1}<\beta$.

\medskip
For the second conclusion:
\[f\in \Lip_{\T}(\alpha;p)\implies \whf\in \ell^{\gamma}(\Z)\]
for $\frac{p}{\alpha p+p-1}<\gamma\leq q$.

\medskip
Consequently Theorem \ref{Ti1} with the second conclusion is an extension of Titchmarsh theorem (cf. Titchmarsh \cite[Theorem 84]{tich:b1}) and Younis \cite[Theorem 2.6]{myou:lpli}) to the setting of compact Lie groups. As it was already mentioned in the introduction the lower bound $\frac{p}{\alpha p+p-1}$ in Theorem \ref{Ti1} is in general sharp as a counterexample can be constructed for $p=2$ in the case $G=\T$.
\end{rem}
For the proof of the second Titchmarsh theorem we will be using the following lemma due to Duren (cf. \cite[p. 101]{du:b1}):

\begin{lem}\label{dulem}
Suppose $c_k\geq 0$ and $0<b<a$. Then
\[\sum\limits_{k=1}^Nk^ac_k=\Ob(N^b)  \mbox{ as }N\rightarrow\infty \]
if and only if
\[\sum\limits_{k=N}^{\infty}c_k=\Ob(N^{b-a})  \mbox{ as }N\rightarrow\infty .
\]
\end{lem}
An extension of this lemma will be proved later as Lemma \ref{dulemlog}.

We can now state our second  main result which extends Theorem B to compact Lie groups and consequently, in view of Lemma \ref{ewrt1}, to compact homogeneous manifolds (Theorem B$'$). 

\begin{thm}\label{Ti2}  Let $0<\alpha\leq 1$ and $f\in L^2(G)$. Then 
$f\in \Lip_{G}(\alpha;2)$ if and only if 
\beq
\sum\limits_{[\xi]\in\Gh, \langle\xi\rangle\geq N}d_{\xi}\|\widehat{f}(\xi)\|_{\HS}^2=\Ob(N^{-2\alpha}) \mbox{ as } N\rightarrow\infty.
\label{t2mg}\eq
\end{thm}
\begin{proof} We first assume that $f\in \Lip_{G}(\alpha;2)$. By \eqref{main2} applied to the case $p=q=2$ we obtain
\beq \sum\limits_{[\xi]\in\Gh, \langle\xi\rangle\leq N}d_{\xi}\langle\xi\rangle^2\|\whf(\xi)\|_{\HS}^2=\Ob(N^{2(1-\alpha)}).\label{t2mg2}\eq

Keeping the notations from the proof of Theorem \ref{Ti1} for $\Phi(N)$ with $\beta=2$ we write
\[\Phi(N)=\sum\limits_{[\xi]\in\Gh, \langle\xi\rangle\leq N}d_{\xi}^{2}\langle\xi\rangle^2\left(\frac{\|\whf(\xi)\|_{\HS}}{\sqrt{d_{\xi}}}\right)^2=\sum\limits_{[\xi]\in\Gh, \langle\xi\rangle\leq N}\sum\limits_{i,j=1}^{d_{\xi}}\langle\xi\rangle^2\left(\frac{\|\whf(\xi)\|_{\HS}}{\sqrt{d_{\xi}}}\right)^2.\]
We put $b_{\xi}=\left(\frac{\|\whf(\xi)\|_{\HS}}{\sqrt{d_{\xi}}}\right)^2$ and in order  to estimate the sum $\sum\limits_{[\xi]\in\Gh, \langle\xi\rangle\leq N}\sum\limits_{i,j=1}^{d_{\xi}}\langle\xi\rangle^2b_{\xi}$ we consider the set $W=\{\jp{\xi}:[\xi]\in\Gh\}$ which is nothing but the range of eigenvalues of the operator $(I-\lap)^{\half}$ since by definition $\jp{\xi}:=(1+\lambda^2_{[\xi]})^{\half}$ where $\lambda^2_{[\xi]}$ are the eigenvalues of $-\lap$. Since $W$ is discrete countable  we can enumerate it as $W=\{\lambda_k:k\in\ene\}$. Of course this enumeration also induces an enumeration of $\Gh$ and by the eigenvalue estimate $\lambda_k\approx k^{\frac 1n}$ (see e.g \cite{ar:pq1}) we obtain
\begin{align}
\sum\limits_{[\xi]\in\Gh, \langle\xi\rangle\leq N}\sum\limits_{i,j=1}^{d_{\xi}}\langle\xi\rangle^2b_{\xi}&=\sum\limits_{\lambda_k\leq N}\sum\limits_{i,j=1}^{d_k}\lambda_k^2b_k\nonumber\\
&\approx \sum\limits_{k\leq N^n}\sum\limits_{i,j=1}^{d_k}k^\frac{2}{n}b_k\nonumber\\
&= \sum\limits_{k\leq N^n}k^\frac{2}{n}c_k,\label{iqwin}
\end{align}
with $c_k=d_k^2b_k$. We now apply \eqref{t2mg2} to get 
\[\sum\limits_{k\leq N^n}k^\frac{2}{n}c_k=\sum\limits_{k=1}^{N_0}k^\frac{2}{n}c_k=\Ob(N_0^{\frac{2(1-\alpha)}{n}}),\]
where $N_0=N^n$.
An application of Duren's Lemma \ref{dulem} gives us
\[\sum\limits_{k=N_0}^{\infty}c_k=\Ob(N_0^{\frac{2(1-\alpha)}{n}-\frac{2}{n}})=\Ob(N_0^{-\frac{2\alpha}{n}}).\]

Hence
\[\sum\limits_{k\geq N^n}c_k=\sum\limits_{\lambda_k\geq N}d_k^2b_k=\Ob(N^{-2\alpha}).\]
Therefore
\
\[\sum\limits_{[\xi]\in\Gh, \langle\xi\rangle\geq N}d_{\xi}\|\widehat{f}(\xi)\|_{\HS}^2=\sum\limits_{[\xi]\in\Gh, \langle\xi\rangle\geq N}d_{\xi}^{2}\left(\frac{\|\whf(\xi)\|_{\HS}}{\sqrt{d_{\xi}}}\right)^2=\Ob(N^{-2\alpha}),\]
 which shows \eqref{t2mg}.
 
We now assume \eqref{t2mg}. By Plancherel theorem and with the notation for $f_h$ as left translation as in the proof of Theorem \ref{Ti1} we have
\beq\|\whf(\xi)\,(\xi(h)-I_{d_{\xi}})\|_{\ell^2(\Gh)}^2= \|f_h-f\|_{L^2(G)}^2.\label{hy1bb}\eq
Following the analysis of the proof of Theorem \ref{Ti1} we note that from the ellipticity argument used to obtain \eqref{main1}, we can now just use the fact that $\sum\limits_{|\alpha|=1,|\beta|=1}\Lambda_{\alpha,\beta,\xi}$ is of the second order. Then we get 
 \beq\|\whf (\xi)(\xi(h)-I_{d_{\xi}})\|_{\HS}^2\leq C|h|^2\langle\xi\rangle^2\|\whf(\xi)\|_{\HS}^2.\label{main1rf}\eq
Therefore the proof is now reduced to estimate
\[ \sum\limits_{[\xi]\in\Gh, \langle\xi\rangle\geq N}d_{\xi}\langle\xi\rangle^2\|\whf(\xi)\|_{\HS}^2.\]
But an analogous argument as in  the above proof of the  ``only if" part and using the same notations yield
\begin{multline*}
\|f_h-f\|_{L^2(G)}^2\leq |h|^2 \sum\limits_{[\xi]\in\Gh, \langle\xi\rangle\geq N}d_{\xi}\langle\xi\rangle^2\|\whf(\xi)\|_{\HS}^2 \\
\approx |h|^2 \sum\limits_{k=N}^{\infty}c_k=
|h|^2 O(N^{2(1-\alpha)})=
\Ob(|h|^{2\alpha}),
\end{multline*}
with $N=\frac{1}{|h|}$,
which concludes the proof.
\end{proof}

We now formulate an application of  Theorem \ref{Ti2} to the regularity of Fourier multipliers on
H\"older spaces mentioned in Corollary \ref{c1q0}. For convenience of the reader we repeat its statement:

\begin{cor}\label{c1q} 
Let $0\leq\gamma<1$ and let $a:\Gh\rightarrow\bigcup\limits_{d\in\ene}\ce^{d\times d}$ be such that $a(\xi)\in \ce^{d_{\xi}\times d_{\xi}}$ for each $\xi$ and
\[\|a(\xi)\|_{op}\leq C\jp{\xi}^{-\gamma}. \]
Let $A$ be the Fourier multiplier with symbol $a$, i.e.
$\widehat{Af}(\xi)=a(\xi)\widehat{f}(\xi)$ for all $\xi\in\Gh$. Then
\[A:\Lip_{G}(\alpha;2)\rightarrow\Lip_{G}(\alpha+\gamma;2) \]
is bounded, for all $\alpha$ such that $0<\alpha<1-\gamma$.
\end{cor}

\begin{proof} Let $f\in \Lip_{G}(\alpha;2)$. Then by Theorem \ref{Ti2} we have
\begin{align*} \sum\limits_{[\xi]\in\Gh, \langle\xi\rangle\geq N}d_{\xi}\|\widehat{Af}(\xi)\|_{\HS}^2&=\sum\limits_{[\xi]\in\Gh, \langle\xi\rangle\geq N}d_{\xi}\|a(\xi)\widehat{f}(\xi)\|_{\HS}^2\\
&\leq\sum\limits_{[\xi]\in\Gh, \langle\xi\rangle\geq N}d_{\xi}\|a(\xi)\|_{op}^2\|\widehat{f}(\xi)\|_{\HS}^2\\
&\leq C\sum\limits_{[\xi]\in\Gh, \langle\xi\rangle\geq N}d_{\xi}\jp{\xi}^{-2\gamma}\|\widehat{f}(\xi)\|_{\HS}^2\\
&\leq CN^{-2\gamma}\sum\limits_{[\xi]\in\Gh, \langle\xi\rangle\geq N}d_{\xi}\|\widehat{f}(\xi)\|_{\HS}^2\\
&=\Ob(N^{-2(\alpha+\gamma)}) \mbox{ as } N\rightarrow\infty.
\end{align*}
Again by Theorem \ref{Ti2} this
implies that $A:\Lip_{G}(\alpha;2)\rightarrow\Lip_{G}(\alpha+\gamma;2) $
is bounded for all $\alpha >0$ such that $\alpha+\gamma<1$.
\end{proof}
As an example we will consider the Lipschitz-Sobolev regularity for Bessel potential operators on compact Lie groups. 
First we observe that if $A=(I-\lap)^{-\frac{\gamma}{2}}$ with $0\leq\gamma<1$, by Corollary  \ref{c1q} we have 
\[\|(I-\lap)^{-\frac{\gamma}{2}}f\|_{\Lip_G(\alpha+\gamma;2)}\leq C\|f\|_{\Lip_G(\alpha;2)}\]
 for all $\alpha$ such that $0<\alpha<1-\gamma$.
 
 \medskip
Hence
\beq\|f\|_{\Lip_G(\alpha+\gamma;2)}\leq C\|(I-\lap)^{\frac{\gamma}{2}}f\|_{\Lip_G(\alpha ;2)}.\label{1qop}
\eq

We can now introduce the Sobolev-Lipschitz space $H^{\gamma}\Lip_G(\alpha,2)$ for every $0\leq\gamma<1$ and $0<\alpha<1-\gamma$ by 
\[H^{\gamma}\Lip_G(\alpha,2):=\{f\in\mathcal{D}'(G):(I-\lap)^{\frac{\gamma}{2}}f\in \Lip_G(\alpha,2)\}.\]
From \eqref{1qop} we obtain:

\begin{cor}
For every $0\leq\gamma<1$ and $0<\alpha<1-\gamma$ we have the continuous embedding
\[H^{\gamma}\Lip_G(\alpha,2)\hookrightarrow \Lip_G(\alpha+\gamma,2) .\]
\end{cor}

\section{Dini-Lipschitz functions}  
\label{SEC:dini-li}

So far we have being dealing with the H\"older-Lipschitz condition on compact homogeneous manifolds in the $L^p$ setting and a characterisation in the $L^2$ case.
 In this section we will consider a different condition, the so-called Dini-Lipschitz condition on $L^2$ and we will generalise the corresponding Titchmarsh theorems (cf. \cite[Theorem 85]{tich:b1}). 
 
We shall first establish an extension of Duren's lemma (cf. \cite[p. 101]{du:b1}), Lemma \ref{dulem} in this paper, adapted to the Dini-Lipschitz condition.
\begin{lem}\label{dulemlog}
Suppose $d\in\ar ,\, c_k\geq 0$ and $0<b<a$. Then
\[\sum\limits_{k=1}^Nk^ac_k=\Ob(N^b(\log N)^d) \;\mbox{ as }\;N\rightarrow\infty \]
if and only if
\[\sum\limits_{k=N}^{\infty}c_k=\Ob(N^{b-a}(\log N)^d) \; \mbox{ as }\;N\rightarrow\infty .
\]
\end{lem}
\begin{proof} We first assume that $\sum\limits_{k=1}^Nk^ac_k=\Ob(N^b(\log N)^d) \mbox{ as }N\rightarrow\infty $, and write
 \beq S_N=\sum\limits_{k=1}^Nk^ac_k.\label{pardlog}\eq
By using the partial summation formula (cf. \cite[Theorem 3.41]{wru:b1}) and \eqref{pardlog} we have
\begin{align*}\sum\limits_{k=N}^{M}c_k=&\sum\limits_{k=N}^{M-1}S_k(k^{-a}-(k+1)^{-a})+S_MM^{-a}-S_{N-1}N^{-a}\\
\leq & C\sum\limits_{k=N}^{M-1}k^b(\log k)^{d}(k^{-a}-(k+1)^{-a}) +CM^{b-a}(\log M)^{d}\\
\leq & C\sum\limits_{k=N}^{M-1}k^{b-a-1}(\log k)^{d} +CM^{b-a}(\log M)^{d},
\end{align*}
where we have also used $k^{-a}-(k+1)^{-a}\leq ck^{-a-1}$ in the last inequality.\\

Letting $M\rightarrow\infty$ the conclusion is reduced to estimate the series $\sum\limits_{k=N}^{\infty}k^{b-a-1}(\log k)^{d}$. We observe that
an integration by parts argument shows that 
\begin{align}\int\limits_{N}^{\infty}t^{b-a-1}(\log t)^{d}dt=& \frac{1}{b-a}\int\limits_{N}^{\infty}(\log t)^{d} \frac{d}{dt}t^{b-a}dt\nonumber\\
=&\frac{t^{b-a}}{b-a}(\log t)^{d} \Bigr|_N^{\infty}-\frac{d}{b-a}\int\limits_{N}^{\infty}t^{b-a-1}(\log t)^{d-1}dt\nonumber\\
=&\frac{N^{b-a}}{a-b}(\log N)^{d}+\frac{d}{a-b}\int\limits_{N}^{\infty}t^{b-a-1}(\log t)^{d-1}dt.\label{ink1q}
\end{align}
We now distinguish two cases: $d\leq 0$ and  $d>0$. First, if $d\leq 0$ we obtain 
\begin{align*}\int\limits_{N}^{\infty}t^{b-a-1}(\log t)^{d}dt=&\frac{N^{b-a}}{a-b}(\log N)^{d}-C_{N,a,b,d}\\
=&O(N^{b-a}(\log N)^{d})  \mbox{ as }N\rightarrow\infty,
\end{align*}
where the number $C_{N,a,b,d}$ is non-negative since $\frac{d}{a-b}\leq 0$ and $t^{b-a-1}(\log t)^{-d-1}>0$. \\

If $d>0$, we take the least positive integer $\ell$ such that $d\leq \ell-1$. An inductive argument applying integration by parts to the integral on the right-hand side  of \eqref{ink1q} gives us 
\[\int\limits_{N}^{\infty}t^{b-a-1}(\log t)^{d}dt=O(N^{b-a}(\log N)^{d})+\frac{d(d-1)\cdots(d-(\ell-1))}{(a-b)^{\ell}}\int\limits_{N}^{\infty}t^{b-a-1}(\log t)^{d-\ell}dt.\]
Once more, the fact that $t^{b-a-1}(\log t)^{d-\ell}\geq 0$ and the choice of $\ell$ implies that
\[d(d-1)\cdots(d-(\ell-1))\leq 0,\]
 which concludes the proof of the ``only if part".\\

Conversely, we assume the second assertion 
 and write 
 \beq R_N=\sum\limits_{k=N}^{\infty}c_k.\label{pardlog2}\eq
Now, the partial summation formula gives
\begin{align*}\sum\limits_{k=1}^{N}k^ac_k=&\sum\limits_{k=1}^{N-1}R_k(k^{a}-(k+1)^a)+R_NN^a\\
\leq & C\sum\limits_{k=1}^{N-1}k^{b-a}(\log k)^{d}(k^{a}-(k+1)^{a}) +CN^{b-a}(\log(N))^{d} N^a\\
\leq & C\sum\limits_{k=1}^{N-1}k^{b-a}(\log k)^{d}k^{a-1} +CN^{b}(\log(N))^{d} \\
= & C\sum\limits_{k=1}^{N-1}k^{b-1}(\log k)^{d} +CN^{b}(\log(N))^{d}.
\end{align*}
The sum $\sum\limits_{k=1}^{N-1}k^{b-1}(\log k)^{d}$ can be estimated by integration in a similar way as we have obtained  \eqref{ink1q}. The inductive argument explained above gives 
\[\int\limits_{1}^{N-1}t^{b-1}(\log t)^{d}dt=O((N-1)^{b}(\log (N-1))^{d})+\frac{d(d-1)\cdots(d-(\ell-1))}{b^{\ell}}\int\limits_{1}^{N-1}t^{b-1}(\log t)^{d-\ell}dt.\]
By using this formula, one can deduce that
\[\sum\limits_{k=1}^{N}k^ac_k=\Ob(N^{b}(\log N)^{d})   \mbox{ as }N\rightarrow\infty .\]
The proof is complete.
\end{proof}
 
We can now establish our first theorem on Dini-Lipschitz functions.
\begin{thm}\label{dinlip0} Let $G/K$ be a compact homogeneous manifold. Let $\alpha\geq 0$ and $d\in\ar$. Then the conditions
\beq\|f(h\cdot)-f(\cdot)\|_{L^2(G/K)}=\Ob\left(|h|^{\alpha}\left(\log\frac{1}{|h|}\right)^{d}\right) \mbox{ as }|h|\rightarrow 0 
\label{dinlip1aa}\eq
and
\beq \sum\limits_{[\xi]\in\Ghz, \langle\xi\rangle\geq N}d_{\xi}\|\widehat{f}(\xi)\|_{\HS}^2=\Ob\left(N^{-2\alpha}(\log N)^{2d}\right) \mbox{ as }N\rightarrow \infty  
\label{dinlip1c0}\eq
are equivalent.
\begin{proof} As in the previous section it will be enough to prove our results for compact Lie groups. First we recall that by \eqref{hy1bb} we have 
\beq\|\whf(\xi)\,(\xi(h)-I_{d_{\xi}})\|_{\ell^2(\Gh)}^2= \|f_h-f\|_{L^2(G)}^2.\label{hy1bbc}\eq
Now if we assume \eqref{dinlip1aa} we obtain
\beq\|\whf(\xi)\,(\xi(h)-I_{d_{\xi}})\|_{\ell^2(\Gh)}^2=\Ob\left(|h|^{2\alpha}\left(\log\frac{1}{|h|}\right)^{2d}\right) \mbox{ as }|h|\rightarrow 0 .\label{h1bbc}\eq
A look at the deduction of \eqref{main2} and by \eqref{h1bbc} shows that for $p=q=2$:
\beq
|h|^2\sum\limits_{[\xi]\in\Gh, \langle\xi\rangle\leq \frac{1}{|h|}}d_{\xi}\langle\xi\rangle^2\|\whf(\xi)\|_{\HS}^2
=\Ob(|h|^{2\alpha}(\log(1/|h|))^{2d})  \mbox{ as }|h|\rightarrow 0 .
\label{main2sw}
\eq
Hence 
\[\sum\limits_{[\xi]\in\Gh, \langle\xi\rangle\leq \frac{1}{|h|}}d_{\xi}\langle\xi\rangle^2\|\whf(\xi)\|_{\HS}^2
=\Ob(|h|^{2\alpha-2}(\log|h|)^{2d})  \mbox{ as }|h|\rightarrow 0 .\]
We now write $N=1/h$. Then
\beq\label{estwf}\sum\limits_{[\xi]\in\Gh, \langle\xi\rangle\leq N}d_{\xi}\langle\xi\rangle^2\|\whf(\xi)\|_{\HS}^2
=\Ob(N^{2-2\alpha}(\log (N))^{2d})  \mbox{ as } N\rightarrow \infty .\eq
An analogous argument as in the proof of Theorem \ref{Ti2} to obtain \eqref{iqwin} and using the same notations lead us to
\[\sum\limits_{[\xi]\in\Gh, \langle\xi\rangle\leq N}d_{\xi}\langle\xi\rangle^2\|\whf(\xi)\|_{\HS}^2\leq\sum\limits_{k\leq N^n}k^\frac{2}{n}c_k=\sum\limits_{k=1}^{N_0}k^\frac{2}{n}c_k=\Ob(N_0^{\frac{2-2\alpha}{n}}(\log N_0^{\frac{1}{n}})^{2d}),\]
where $N_0=N^{\frac{1}{n}}$. An application of Lemma \ref{dulemlog} gives us 
\[\sum\limits_{k=N_0}^{\infty}c_k=\Ob(N_0^{\frac{2-2\alpha}{n}-\frac{2}{n}}(\log N_0^{\frac{1}{n}})^{2d})  \mbox{ as } N_0\rightarrow \infty .\]
Therefore 
\[ \sum\limits_{[\xi]\in\Gh, \langle\xi\rangle\geq N}d_{\xi}\|\widehat{f}(\xi)\|_{\HS}^2=\Ob(N^{-2\alpha}(\log N)^{2d})   \mbox{ as } N\rightarrow \infty .\]

A look at the proof above shows that the converse also holds.
\end{proof}
\end{thm}

To end this work we now give a Dini-Lipschitz version of Theorem \ref{Ti1}.

\begin{thm}\label{Ti1dl} 
Let $G/K$ be a compact homogeneous manifold of dimension $n$. Let $0<\alpha\leq 1$, $d\in\ar$, $1<p\leq 2$, and let $q$ be such that $\sdual$. 
Suppose that
\beq\|f(h\cdot)-f(\cdot)\|_{L^p(G/K)}=\Ob\left(|h|^{\alpha}\left(\log\frac{1}{|h|}\right)^{d}\right) \mbox{ as }|h|\rightarrow 0.
\label{dinlip1a}\eq
Then we have
\[\reallywidehat{(I-\lap)^{\half}f}\in \ell^{\beta}(\Ghz),\]
provided that either
\[\frac{n}{\alpha+n-\frac{n}{p}-1}< \beta\leq q \,\,\,\mbox{ for any }\,\,d\in\ar\]
or 
\[\frac{n}{\alpha+n-\frac{n}{p}-1}\leq \beta\leq q \,\,\,\mbox{ for  }\,\,d\leq 0.\]
Consequently, in both cases we have
\[\whf\in \ell^{\gamma}(\Ghz)\;\textrm{ for }\;\frac{np}{\alpha p+np-n}<\gamma\leq q.\] 
\end{thm} 
\begin{proof} Arguing as in the proof of Theorem \ref{Ti1}, under assumptions of Theorem \ref{Ti1dl} 
instead of \eqref{miqj3} we arrive at
\beq\sum\limits_{[\xi]\in\Gh, \langle\xi\rangle\leq N}d_{\xi}^{2}\left(\frac{\langle\xi\rangle\|\whf(\xi)\|_{\HS}}{\sqrt{d_{\xi}}}\right)^{\beta}=O(N^{(1-\alpha)\beta+n(1-\frac{\beta}{q})}(\log N)^{d\beta}) \mbox{ as } N\rightarrow\infty.\label{miqj3a}\eq
So we get
\[\reallywidehat{(I-\lap)^{\half}f}\in \ell^{\beta}(\Ghz)\]
provided that $(1-\alpha)\beta+n(1-\frac{\beta}{q})<0$, or $(1-\alpha)\beta+n(1-\frac{\beta}{q})\leq 0$ and $d\leq 0$. The last conclusion follows in the same way as in Theorem \ref{Ti1}.
\end{proof}


\end{document}